\documentclass{amsart}
\usepackage{amsfonts}
\usepackage{amsmath}
\usepackage{amssymb}
\usepackage{amsthm}

\usepackage{graphicx}
\usepackage{capt-of}
\usepackage{tikz}
\usetikzlibrary{patterns}
\usetikzlibrary{decorations.markings}
\usepackage{multirow}
\usepackage{array} 
\newcolumntype{C}{>{$}c<{$}}

\newtheorem{theorem}{Theorem}

\newtheorem{lemma}[theorem]{Lemma}
\newtheorem{rem}[theorem]{Remark}
\newtheorem{exmp}[theorem]{Example}

\begin{document}

\title{On the graph of non-degenerate linear $[n,2]_2$ codes}
\author{Mark Pankov}
\address{Faculty of Mathematics and Computer Science, University of Warmia and Mazury, S{\l}oneczna 54, Olsztyn, Poland}
\email{pankov@matman.uwm.edu.pl}

\maketitle

\begin{abstract}
Consider the Grassmann graph of $k$-dimensional subspaces of an $n$-dimensional vector space over the $q$-element field,
$1<k<n-1$.
Every automorphism of this graph is induced 
by a semilinear automorphism of the corresponding vector  space or a semilinear isomorphism to the dual vector space;
the second possibility is realized only for $n=2k$.
Let $\Gamma(n,k)_q$ be the subgraph of the Grassman graph formed by all non-degenerate linear $[n,k]_q$ codes.
If $q\ge 3$ or $k\ge 3$,
then every isomorphism of  $\Gamma(n,k)_{q}$ to a subgraph of the Grassmann graph
can be uniquely extended to an automorphism of the Grassmann graph.
For $q=k=2$ there is an isomorphism of  $\Gamma(n,k)_{q}$ to a subgraph of the Grassmann graph which does not have this property.
In this paper, we show that such exceptional isomorphism is unique up to an automorphism of the Grassmann graph.
\end{abstract}

\section{Introduction}

The $(n,k,q)$-Grassmann graph is formed by all $k$-dimen\-sional subspaces of  an $n$-dimensional vector space over the field of $q$ element.
We suppose that $1<k<n-1$ (for $k=1,n-1$ the graph is complete).
This is a classical example of distance regular graph \cite{BCN} and
Grassmann graphs (not necessarily over finite fields) are related to Tits buildings of general linear groups \cite{BC,Pankov1,Pasini,Shult}.

The $(n,k,q)$-Grassmann graph can be identified with the graph of linear $[n,k]_q$ codes 
such that two codes are connected by an edge if they have the maximal number of common codewords.
In practice, each useful linear code is non-degenerate, i.e. a generator matrix of such a code does not contain zero columns.
The graph of non-degenerate linear $[n,k]_q$ codes (the induced subgraph of the $(n,k,q)$-Grassmann graph) is investigated in  \cite{KP1,KP2,PankJACTA}.
In contrast to the $(n,k,q)$-Grassmann graph,  the structure of this graph essentially depends on the parameters  $n,k,q$ (structural properties holding for some triples $n,k,q$ fail for others).
Induced subgraphs of the $(n,k,q)$-Grassmann graph corresponding to different types of linear codes (projective and simplex codes, codes with lower bounded minimal dual distance) are considered in \cite{CGK,KPP,KP3}.

By Chow's theorem \cite{Chow}, every automorphism of a Grassmann graph (over an arbitrary field) is induced  by a semilinear automorphism of  
the corresponding vector space or a semilinear isomorphism to the dual vector space
(the second possibility is realized only in the case when the dimension of subspaces is the half of the dimension of the vector space).
The main result of \cite{PankJACTA} concerns isomorphisms of the graph of non-degenerate linear $[n,k]_q$ codes to subgraphs of the $(n,k,q)$-Grassmann graph.
Every such  isomorphism can be uniquely extended to an automorphism of the $(n,k,q)$-Grassmann graph if $q\ge 3$ or $k\ge 3$
and there is a non-extendable isomorphism if  $q=k=2$. 
In Section 4, we recall some arguments from \cite{PankJACTA} and explain why they do not work for the case when $q=k=2$. 

In the present paper, we show that a non-extendable isomorphism of the graph of non-degenerate linear $[n,2]_2$ codes to a subgraph of the $(n,2,2)$-Grassmann graph 
is unique up to an automorphism of this Grassmann graph (Theorem \ref{theorem-main}).
Note that this isomorphism is induced by a non-injective morphism of the corresponding projective space (Remark \ref{rem-morph}).

\section{Grassmann graphs and graphs of non-degenerate linear codes}
Let $V$ be an $n$-dimensional vector space over a field ${\mathbb F}$ (the field is not assumed to be finite)
and let ${\mathcal G}_{k}(V)$ be the Grassmannian  formed by $k$-dimensional subspaces of $V$.
We say that two $k$-dimensional subspaces are {\it adjacent} if their intersection is $(k-1)$-dimensional 
(then the sum of these subspaces is $(k+1)$-dimensional).
The {\it Grassmann graph} $\Gamma_{k}(V)$ is the simple graph whose vertex set is ${\mathcal G}_{k}(V)$ and 
two vertices ($k$-dimensional subspaces) are connected by an edge if they are adjacent.
The Grassmann graph is connected. If $k=1,n-1$, then this graph is complete.

The description of automorphisms of Grassmann graphs is based on the concept of semilinear automorphism.
A map $l:V\to V$ is called {\it semilinear} if 
$$l(x+y)=l(x)+l(y)$$
for all $x,y\in V$ and there is an endomorphism $\sigma$ of the field ${\mathbb F}$
such that 
$$l(ax)=\sigma(a)l(x)$$
for all $a\in {\mathbb F}$ and $x\in V$.
We say that $l$ is a {\it semilinear automorphism} of the vector space $V$
if $l$ is bijective and $\sigma$ is an automorphism of the field ${\mathbb F}$.
Every semilinear automorphism of $V$ induces an automorphism of the Grassmann graph $\Gamma_{k}(V)$.

Let $\omega$ be a non-degenerate symmetric form on $V$.
For every subspace $X\subset V$ we denote by $X^{\perp}_{\omega}$ the orthogonal complement of $X$ corresponding to the form $\omega$.
Then $X,Y\in {\mathcal G}_k(V)$ are adjacent if and only if $X^{\perp}_{\omega},Y^{\perp}_{\omega}\in {\mathcal G}_{n-k}(V)$ are adjacent.
Therefore, the orthocomplementary map $X\to X^{\perp}_{\omega}$ defines an isomorphism between the graphs $\Gamma_{k}(V)$ and $\Gamma_{n-k}(V)$
which is an automorphism of $\Gamma_{k}(V)$ if $n=2k$.

Chow's theorem \cite{Chow} states that every automorphism of $\Gamma_{k}(V)$, $1<k<n-1$ is induced by a semilinear automorphism of $V$
or $n=2k$ and it is the composition of the orthocomplementary map and a graph automorphism  induced by a semilinear automorphism of $V$.
If $k=1,n-1$, then any two distinct $k$-dimensional subspaces of $V$ are adjacent
and every bijective transformation of ${\mathcal G}_{k}(V)$ is an automorphism of $\Gamma_{k}(V)$.

Now, we suppose that ${\mathbb F}={\mathbb F}_{q}$ is the field of $q$ elements and $V={\mathbb F}^n$.
The vectors 
$$e_{1}=(1,0,\dots,0),e_2=(0,1,0,\dots,0),\dots,e_{n}=(0,\dots,0,1)$$
form the standard basis of $V$.
Denote by $C_{i}$ the kernel of the $i$-th coordinate functional $(x_{1},\dots,x_{n})\to x_{i}$; this is the hyperplane of $V$ spanned by all $e_j$ with $j\ne i$.

Every $k$-dimensional subspace of $V$ is a {\it linear} $[n,k]_{q}$ {\it code}.
Such a code $C$ is {\it non-degenerate} if
the restriction of every coordinate functional to $C$ is non-zero, in other words, there is no coordinate hyperplane $C_{i}$ containing $C$ \cite{TVN}.
The set of all non-degenerate linear $[n,k]_{q}$ codes ${\mathcal C}(n,k)_{q}$ is obtained by removing all ${\mathcal G}_{k}(C_{i})$ from ${\mathcal G}_{k}(V)$.

The graph of non-degenerate linear $[n,k]_q$ codes $\Gamma(n,k)_{q}$ is the subgraph of $\Gamma_{k}(V)$ induced by ${\mathcal C}(n,k)_{q}$.
This graph  is connected \cite{KP1}.
If $1<k<n-1$, then every automorphism of $\Gamma(n,k)_{q}$ is induced by a monomial semilinear automorphism of $V$
(a semilinear automorphism sending every $e_i$ to a scalar multiple of $e_j$), see \cite{KP2}.
Non-identity automorphisms of the field ${\mathbb F}_q$ exist only in the case when $q$ is a composite number, i.e.
every semilinear automorphism of $V$ is linear if $q$ is prime.

Suppose that $1<k<n-1$. The main result  of \cite{PankJACTA} states that every isomorphism of  $\Gamma(n,k)_{q}$ to a (not necessarily induced) subgraph of $\Gamma_{k}(V)$ 
can be uniquely extended to an automorphism of $\Gamma_{k}(V)$ if $q\ge 3$ or $k\ge 3$  and there is an isomorphism of $\Gamma(n,2)_{2}$ to a subgraph of $\Gamma_{2}(V)$
which does not have this property. An example of such non-extendable isomorphism will be presented in the next section. 

\begin{rem}{\rm
If $k=1,n-1$, then every injection of ${\mathcal C}(n,k)_{q}$ to ${\mathcal G}_{k}(V)$ is an isomorphism of $\Gamma(n,k)_{q}$ to a subgraph of $\Gamma_{k}(V)$;
it can be extended to a bijective transformation of ${\mathcal G}_{k}(V)$, but such an extension is not unique.
Recall that every bijective transformation of ${\mathcal G}_{k}(V)$ is an automorphism of $\Gamma_{k}(V)$ if $k=1,n-1$.
}\end{rem}

\section{Result}
In this section, we suppose that ${\mathbb F}={\mathbb F}_2$ is the field of two elements 
and consider the graph of non-degenerate linear $[n,2]_2$ codes $\Gamma(n,2)_2$, $n\ge 4$.

Let $Q$ be the $1$-dimensional subspace of $V$ containing the vector $(1,\dots,1)$.
For every proper subset $I\subset \{1,\dots,n\}$ we denote by $P_I$ 
the $1$-dimensional subspace of $V$ containing the vector whose $i$-coordinate is $1$ if $i\in I$ and $0$ if $i\not\in I$. 
For $I^c=\{1,\dots,n\}\setminus I$ we will write $P^I$ instead of $P_{I^c}$ 
when it will be convenient. Also, if $I=\{i_1,\dots,i_m\}$, then we write $P_{i_1,\dots,i_m}$ and $P^{i_1,\dots,i_m}$ instead of $P_I$ and $P^I$, respectively.

Let $H$ be the hyperplane of $V$ spanned by $P^1,\dots,P^{n-1}$.
Then $Q\not\subset H$ (otherwise, $H$ contains all $P_i$ which is impossible). 
For every proper $I\subset \{1,\dots,n\}$ the $2$-dimensional subspace $P_I+P^I$ 
is an element of ${\mathcal C}(n,2)_2$ containing $Q$ which implies that precisely one of $P_I,P^I$ is contained in $H$.
As in \cite[Section 7]{PankJACTA}, we decompose ${\mathcal C}(n,2)_2$ in three subsets ${\mathcal A}, {\mathcal B}, {\mathcal C}$ as follows:
\begin{enumerate}
\item[$\bullet$] ${\mathcal A}$ consists of all elements of ${\mathcal C}(n,2)_2$ containing $Q$;
\item[$\bullet$] ${\mathcal B}$ is formed by all elements of ${\mathcal C}(n,2)_2$ contained in $H$;
\item[$\bullet$] ${\mathcal C}={\mathcal C}(n,2)_2\setminus({\mathcal A}\cup {\mathcal B})$.
\end{enumerate}

\begin{exmp}\label{exmp1}{\rm
The graph $\Gamma(4,2)_{2}$ consists of $13$ vertices.
The set ${\mathcal A}$ is formed by the following $7$ elements $Q+P$, where $P$ is a $1$-dimensional subspace of $H$:
$$Q+P^1=\left[
\begin{array}
{@{\hspace{2pt}}c@{\hspace{2pt}}c@{\hspace{2pt}}c@{\hspace{2pt}}c@{\hspace{2pt}}}
1&1&1&1\\
0&1&1&1\\
1&0&0&0\\
\end{array}
\right],
Q+P^2=\left[
\begin{array}
{@{\hspace{2pt}}c@{\hspace{2pt}}c@{\hspace{2pt}}c@{\hspace{2pt}}c@{\hspace{2pt}}}
1&1&1&1\\
1&0&1&1\\
0&1&0&0\\
\end{array}
\right],
Q+P^3=\left[
\begin{array}
{@{\hspace{2pt}}c@{\hspace{2pt}}c@{\hspace{2pt}}c@{\hspace{2pt}}c@{\hspace{2pt}}}
1&1&1&1\\
1&1&0&1\\
0&0&1&0\\
\end{array}
\right],
Q+P_4=\left[
\begin{array}
{@{\hspace{2pt}}c@{\hspace{2pt}}c@{\hspace{2pt}}c@{\hspace{2pt}}c@{\hspace{2pt}}}
1&1&1&1\\
0&0&0&1\\
1&1&1&0\\
\end{array}
\right],$$
$$
Q+P_{1,2}=\left[
\begin{array}
{@{\hspace{2pt}}c@{\hspace{2pt}}c@{\hspace{2pt}}c@{\hspace{2pt}}c@{\hspace{2pt}}}
1&1&1&1\\
1&1&0&0\\
0&0&1&1\\
\end{array}
\right],\;\;
Q+P_{1,3}=\left[
\begin{array}
{@{\hspace{2pt}}c@{\hspace{2pt}}c@{\hspace{2pt}}c@{\hspace{2pt}}c@{\hspace{2pt}}}
1&1&1&1\\
1&0&1&0\\
0&1&0&1\\
\end{array}
\right],\;\;
Q+P_{2,3}=\left[
\begin{array}
{@{\hspace{2pt}}c@{\hspace{2pt}}c@{\hspace{2pt}}c@{\hspace{2pt}}c@{\hspace{2pt}}}
1&1&1&1\\
0&1&1&0\\
1&0&0&1\\
\end{array}
\right]$$
(the rows of each matrix are non-zero vectors of the corresponding subspace).
The set ${\mathcal B}$ consists of the following elements $P^i+P^j$, where $i,j\in \{1,2,3\}$ are distinct:
$$
P^1+P^2=\left[
\begin{array}
{@{\hspace{2pt}}c@{\hspace{2pt}}c@{\hspace{2pt}}c@{\hspace{2pt}}c@{\hspace{2pt}}}
0&1&1&1\\
1&0&1&1\\
1&1&0&0\\
\end{array}
\right],
P^1+P^3=\left[
\begin{array}
{@{\hspace{2pt}}c@{\hspace{2pt}}c@{\hspace{2pt}}c@{\hspace{2pt}}c@{\hspace{2pt}}}
0&1&1&1\\
1&1&0&1\\
1&0&1&0\\
\end{array}
\right],
P^2+P^3=\left[
\begin{array}
{@{\hspace{2pt}}c@{\hspace{2pt}}c@{\hspace{2pt}}c@{\hspace{2pt}}c@{\hspace{2pt}}}
1&0&1&1\\
1&1&0&1\\
0&1&1&0\\
\end{array}
\right]
$$
Finally, ${\mathcal C}$ is formed by the following elements $P^i+P^4$ with $i\in \{1,2,3\}$:
$$P^1+P^4=\left[
\begin{array}
{@{\hspace{2pt}}c@{\hspace{2pt}}c@{\hspace{2pt}}c@{\hspace{2pt}}c@{\hspace{2pt}}}
0&1&1&1\\
1&1&1&0\\
1&0&0&1\\
\end{array}\right],
P^2+P^4=\left[
\begin{array}
{@{\hspace{2pt}}c@{\hspace{2pt}}c@{\hspace{2pt}}c@{\hspace{2pt}}c@{\hspace{2pt}}}
1&0&1&1\\
1&1&1&0\\
0&1&0&1\\
\end{array}\right],
P^3+P^4=\left[
\begin{array}{@{\hspace{2pt}}c@{\hspace{2pt}}c@{\hspace{2pt}}c@{\hspace{2pt}}c@{\hspace{2pt}}}
1&1&0&1\\
1&1&1&0\\
0&0&1&1\\
\end{array}
\right].
$$
Note that every $P^i+P^j$ contains $P_{i,j}$.}\end{exmp}

Let $X\in {\mathcal C}$. It intersects $H$ in a $1$-dimensional subspace $P_T$
and we suppose that $P_I,P_J$  are the remaining two $1$-dimensional subspaces of $X$.
Then
$$
I\cup J=\{1,\dots,n\}
$$
(since $X\in {\mathcal C}(n,2)_2$)
and $$I\cap J\ne \emptyset$$
(otherwise, $J=I^c$ and $X$ contains $Q$ which is impossible).
It is easy to see  that
$$T=(I\cap J)^c,$$ 
i.e. $X$ intersects $H$ in $P^{I\cap J}$.
The hyperplane $H$ does not contain $P_I,P_J$ and, consequently, $P^I,P^J$ both are contained in $H$. 
Denote by $X^c$ the $2$-dimensional subspace $P^I+P^J$.
It is clear that $X^c\subset H$.
Since $$I^c\cap J^c=\emptyset,$$
the $1$-dimensional subspace of $X^c$ distinct from $P^I$ and $P^J$  is $P_{I^c\cup J^c}=P^{I\cap J}$.
Therefore, $X^c$ is contained in $C_i$ for every $i\in I\cap J$, i.e.
$X^c$ does not belong to ${\mathcal C}(n,2)_2$.
The $1$-dimensional subspaces of $X$ and $X^c$ are 
$$P_I,P_J, P^{I\cap J}\;\mbox{ and }\;P^I,P^J, P^{I\cap J}$$ 
(respectively), i.e. $P^{I\cap J}$ is the intersection of $X$ and $X^c$.

\begin{exmp}\label{exmp2}{\rm
As in Example \ref{exmp1}, we suppose that $n=4$. Then
$$(P^1+P^4)^c=P^1+P_4=\left[
\begin{array}
{@{\hspace{2pt}}c@{\hspace{2pt}}c@{\hspace{2pt}}c@{\hspace{2pt}}c@{\hspace{2pt}}}
0&1&1&1\\
0&0&0&1\\
0&1&1&0\\
\end{array}\right],$$
$$(P^2+P^4)^c=P^2+P_4=\left[
\begin{array}
{@{\hspace{2pt}}c@{\hspace{2pt}}c@{\hspace{2pt}}c@{\hspace{2pt}}c@{\hspace{2pt}}}
1&0&1&1\\
0&0&0&1\\
1&0&1&0\\
\end{array}
\right],$$
$$(P^3+P^4)^c=P^3+P_4=\left[
\begin{array}{@{\hspace{2pt}}c@{\hspace{2pt}}c@{\hspace{2pt}}c@{\hspace{2pt}}c@{\hspace{2pt}}}
1&1&0&1\\
0&0&0&1\\
1&1&0&0\\
\end{array}
\right].
$$
Note that for every $i\in \{1,2,3\}$ the subspaces $P^i+P^4$ and $(P^i+P^4)^c=P^i+P_4$ contain $P_{i,4}$ and $P^{i,4}$, respectively.
}\end{exmp}

Consider the map $$h:{\mathcal C}(n,2)_2\to {\mathcal G}_2(V)$$
which leaves fixed  every element of ${\mathcal A}\cup {\mathcal B}$ and sends every $X\in {\mathcal C}$ to $X^c$.
This map is injective (since $X^c\not\in {\mathcal C}(n,2)_2$ for every $X\in {\mathcal C}$ and $X^c\ne Y^c$ for distinct $X,Y\in {\mathcal C}$).
The map $h$ preserves the adjacency  relation only in one direction \cite[Section 7]{PankJACTA}, i.e. 
it sends adjacent elements of ${\mathcal C}(n,2)_2$ to adjacent elements of ${\mathcal G}_2(V)$ and there are non-adjacent $X,Y\in {\mathcal C}(n,2)_2$ such that $h(X),h(Y)$ are adjacent.
In the subgraph of $\Gamma_2(V)$ induced by the image of $h$ we remove the edge connecting $h(X)$ and $h(Y)$ if $X,Y\in {\mathcal C}(n,2)_2$ are not adjacent.
We obtain a certain subgraph of $\Gamma_2(V)$ which is not an induced subgraph. 
The map $h$ is an isomorphism of $\Gamma(n,2)_2$ to this subgraph.
Since $h$ is adjacency preserving only in one direction, it cannot be extended to an automorphism of $\Gamma_2(V)$.

\begin{theorem}\label{theorem-main}
If an isomorphism $f$ of $\Gamma(n,2)_{2}$ to a subgraph of $\Gamma_{2}(V)$
cannot be extended to an automorphism of $\Gamma_{2}(V)$, then 
$f=gh$ for a certain automorphism $g$ of $\Gamma_{2}(V)$.
\end{theorem}

\begin{rem}\label{rem-morph}{\rm
A {\it morphism} (a lineation) of a projective space is a transformation of the set of points sending lines to subsets of lines.
A morphism is an automorphism (a collineation) if it is bijective and the inverse map also is a morphism.
By the Fundamental Theorem of Projective Geometry, every automorphism of a projective space (associated to a vector space) is induced 
by a semilinear automorphism of the corresponding vector space. 
A morphism is {\it non-degenerate} if its image is not contained in a line and the image of every line contains at least three points.
Every non-degenerate morphism is induced by a generalized semilinear map associated to a place \cite{Faure}. 
Degenerate morphisms of projective spaces are not determined.
If a projective space is over the field of two elements, then every line contains precisely three points and, consequently, every non-injective morphism is degenerate. 
Consider the transformation of ${\mathcal G}_1(V)$ which leaves fixed $Q$ and all $P_I\subset H$ and sends every $P_I\not\subset H$ to $P^I\subset H$. 
This is a non-injective morphism of the projective space associated to $V$. This morphism induces $h$. 
}\end{rem}

\section{Maximal cliques}
A subset in the vertex set of a graph is called a {\it clique} if any two distinct vertices from this subset are connected by an edge.
A clique ${\mathcal X}$ is  {\it maximal} if every clique containing ${\mathcal X}$ coincides with ${\mathcal X}$.
Every maximal clique of the Grassmann graph $\Gamma_{k}(V)$, $1<k<n-1$ (${\mathbb F}$ is an arbitrary field) is of one of the following types:
\begin{enumerate}
\item[$\bullet$] the {\it star} ${\mathcal S}(X)$, where $X$ is a $(k-1)$-dimensional subspace of $V$; 
it consists of all $k$-dimensional subspaces containing $X$;
\item[$\bullet$] the {\it top} ${\mathcal G}_{k}(Y)$, where $Y$ is a $(k+1)$-dimensional subspace of $V$.
\end{enumerate}

Suppose that ${\mathbb F}={\mathbb F}_q$ is the field of $q$ elements and $1<k<n-1$.
The intersection of ${\mathcal C}(n,k)_q$ with a maximal clique of $\Gamma_k(V)$ is a clique of $\Gamma(n,k)_q$,
but such a clique need not to be maximal.  The intersection of ${\mathcal C}(n,k)_q$ with a star or a top of $\Gamma_k(V)$
will be called a {\it star} or, respectively, a {\it top} of $\Gamma(n,k)_q$ only when it is a maximal clique of $\Gamma(n,k)_q$.

For every $(k-1)$-dimensional subspace $X\subset V$ we define 
$${\mathcal S}^c(X)={\mathcal C}(n,k)_q\cap {\mathcal S}(X).$$
If $q\ge 3$, then ${\mathcal S}^c(X)$ is a star of $\Gamma(n,k)_q$ for each $(k-1)$-dimensional subspace $X$ \cite[Proposition 1]{KP2}.
For $q=2$ the same holds if and only if the number of coordinate hyperplanes $C_i$ containing $X$ is not greater than $n-k-1$ \cite[Proposition 2]{KP2}.
We have ${\mathcal S}^c(X)={\mathcal S}(X)$ only for $X\in {\mathcal C}(n,k-1)_q$; such a clique is said to be a {\it maximal star} of $\Gamma(n,k)_q$.

\begin{exmp}\label{exmp3}{\rm
Suppose that $q=k=2$.
For a $1$-dimensional subspace $P\subset V$ the set ${\mathcal S}^c(P)$ is a star of $\Gamma(n,2)_2$ if and only if $P=Q$ or $P=P_I$ with $|I|\ge 3$.
Hence ${\mathcal S}(Q)$ is the unique maximal star of $\Gamma(n,2)_2$;
it coincides with the set ${\mathcal A}$ considered in the previous section.
}\end{exmp}

\begin{rem}{\rm
Tops of $\Gamma(n,k)_q$ exist. For example, the maximal clique of $\Gamma(n,2)_2$ formed by
$$Q+P^1,\;Q+P^2,\; Q+P_{1,2},\;P^1+P^2$$
is a top.
It must be pointed out that every star of $\Gamma(n,k)_q$ is not a top and vice versa \cite[Proposition 5]{PankJACTA}.
}\end{rem}

We recall some reasonings from \cite{PankJACTA}.

Let $f$ be an isomorphism of $\Gamma(n,k)_q$ to a subgraph of $\Gamma_k(V)$, $1<k<n-1$.
By \cite[Lemma 9]{PankJACTA},  $f$ sends every maximal star of $\Gamma(n,k)_q$ to a star of $\Gamma_k(V)$
or $n=2k$ and $f$ transfers all stars of $\Gamma(n,k)_q$ to tops of $\Gamma_k(V)$.
In the second case, the composition of the orthocomplementary map (which is an automorphism of $\Gamma_k(V)$ if $n=2k$) and $f$ sends  every maximal star of $\Gamma(n,k)_q$ to a star of $\Gamma_k(V)$.
For this reason, it is sufficient to consider the first case only. 
In this case, there is a sequence of maps
$$f_i:{\mathcal C}(n,i)_q\to {\mathcal G}_i(V),\;\;\;i=k,\dots,1$$
such that $f_k=f$, each $f_i$ is an isomorphism of $\Gamma(n,i)_q$ to a subgraph of $\Gamma_i(V)$ and
$$f_i({\mathcal S}(X))={\mathcal S}(f_{i-1}(X))$$
for every $X\in {\mathcal C}(n,i-1)_q$ with $i\ge 2$.
By \cite[Lemma 10]{PankJACTA},
each $f_i$, $i\ge 2$ transfers every star of $\Gamma(n,k)_q$ to a clique contained in a star of $\Gamma_k(V)$; such a star of $\Gamma_k(V)$ is unique  
(since $f_i$ is injective and the intersection of two distinct stars of $\Gamma_k(V)$ contains at most one element). 

If $q\ge 3$, then every $f_j$, $j\in \{1,\dots,k-1\}$ can be extended to a transformation $g_j$ of ${\mathcal G}_{j}(V)$ such that
$$f_{j+1}({\mathcal S}^c(X))\subset  {\mathcal S}^c(g_j(X))$$
for every $X\in {\mathcal G}_{j}(V)$.
The map $g_1$ is an automorphism of the projective space associated to $V$ and, consequently,
by the Fundamental Theorem of Projective Geometry, it is induced by a semilinear automorphism of $V$
which implies that $f$ can be uniquely extended to an automorphism of $\Gamma_k(V)$ \cite[Section 5]{PankJACTA}.

Suppose that $q=2$. 
In this case, ${\mathcal C}(n,1)_2=\{Q\}$.
Denote by ${\mathcal G}'_1(V)$ the set consisting of $Q$ and all $1$-dimensional subspaces $P_I$ satisfying $|I|\ge 3$.
Then ${\mathcal S}^c(P)$ is a star of $\Gamma(n,2)_2$ if and only if $P\in {\mathcal G}'_1(V)$.
For every $P\in {\mathcal G}'_1(V)$ there is a unique $1$-dimensional subspace $g_1(P)$ such that
$$f_2({\mathcal S}^c(P))\subset {\mathcal S}(g_1(P)).$$
If $k\ge 3$, then the map 
$$g_1:{\mathcal G}'_1(V)\to {\mathcal G}_1(V)$$
can be uniquely extended to an automorphism of the projective space associated to $V$ which implies $f$ can be uniquely extended to an automorphism of $\Gamma_k(V)$ \cite[Section 6]{PankJACTA}.
For $k=2$ this fails: $g_1$ cannot be extended to an automorphism of the projective space associated to $V$ if $f=h$ ($h$ as in the previous section);
furthermore, $g_1$ is non-injective if $n\ge 5$ and $f=h$.

\section{Proof of Theorem \ref{theorem-main}}
Let $f$ be an isomorphism of $\Gamma(n,2)_2$, $n\ge 4$ to a subgraph of $\Gamma_2(V)$.
By the previous section,
we can assume that $f$ sends the unique maximal star of $\Gamma(n,2)_2$  to a star of $\Gamma_2(V)$.
Then it transfers every star of $\Gamma(n,2)_2$ to a clique contained in a star of $\Gamma_2(V)$ and induces the map $g_1$ of ${\mathcal G}'_1(V)$ to ${\mathcal G}_1(V)$ such that
$$f({\mathcal S}^c(P))\subset {\mathcal S}(g_1(P))$$
for every $P\in {\mathcal G}'_1(V)$.

For every proper $I\subset \{1,\dots,n\}$ we define
$$A_I=Q+P_I.$$
Then $A_I=Q+P_I=Q+P^I=A_{I^c}$ and one of $I,I^c$ is contained in $\{1,\dots,n-1\}$.
The maximal star ${\mathcal A}={\mathcal S}(Q)$ is formed by all $A_I$ such that $I$ is a non-empty subset of $\{1,\dots,n-1\}$.
If $I=\{i_1,\dots,i_m\}$, then we write $A_{i_1,\dots,i_m}$ instead of $A_I$.

For every non-empty $I\subset \{1,\dots,n-1\}$ we denote by $S_I$ the subspace of $V$ spanned by all $f(A_i)$ with $i\in I$.
As above, we will use the symbol $S_{i_1,\dots,i_m}$ instead of $S_I$ if $I=\{i_1,\dots,i_m\}$.

\begin{lemma}\label{lemma1}
$f(A_I) \subset S_I$ for every non-empty $I\subset \{1,\dots,n-1\}$.
\end{lemma}

\begin{proof} We prove the statement by induction.
The case $|I|=1$ is trivial: $f(A_i)=S_i$ for every $i\in \{1,\dots,n-1\}$.
Suppose that $m$ is an integer such that $1\le m<n-1$ and the inclusion $f(A_I) \subset S_I$ holds for every $I\subset \{1,\dots,n-1\}$
containing not greater than $m$ elements.
Consider an $(m+1)$-element subset $I\subset \{1,\dots n-1\}$.
We have
$$f(A_J)\subset S_J\subset S_I$$
for every proper $J\subset I$.
If $i\in I$ and $J=I\setminus\{i\}$, then 
$$P^i+P^J\in {\mathcal C}(n,2)_2$$ is adjacent to $A_i$ and $A_J$
which implies that $f(P^i+P^J)$ is adjacent to $f(A_i)$ and $f(A_J)$.
The map $f$ is injective,
$$f({\mathcal A})={\mathcal S}(g_1(Q))\;\mbox{ and }\;P^i+P^J\not\in {\mathcal A},$$
i.e. $f(P^i+P^J)$ does not contain $g_1(Q)$.
On the other hand, $g_1(Q)$ is contained in both $f(A_i)$ and $f(A_J)$.
This means that $f(P^i+P^J)$ intersects 
$$f(A_i)\subset S_I\;\mbox{ and }\;f(A_J)\subset S_I$$ in distinct $1$-dimensional subspaces.
Therefore, 
$$f(P^i+P^J)\subset S_I.$$
Observe that $P^i+P^J$ contains $P_I$ (since $I=J\cup \{i\}$) which means that $P^i+P^J$  is adjacent to $A_I$ and, consequently,
$f(P^i+P^J)$ is adjacent to $f(A_I)$. 
So, $f(A_I)$ intersects $S_I$ in two distinct $1$-dimensional subspaces: one of them is $g_1(Q)$ and the other is the intersection of  $f(A_I)$ with $f(P^i+P^J)$.
Hence $f(A_I)\subset S_I$.
\end{proof}

By Lemma \ref{lemma1}, the subspace $S_{1,\dots,n-1}$ contains every element of $f({\mathcal A})$.
Since $f({\mathcal A})$ is a start of $\Gamma_2(V)$,  we obtain that $S_{1,\dots,n-1}=V$.

The equality $f({\mathcal A})={\mathcal S}(g_1(Q))$ shows that $g_1(Q)\ne g_1(P)$ for every $P\in {\mathcal G}'_1(V)$ distinct from $Q$
(since $f$ is injective).
We have 
$${\mathcal A}\cap {\mathcal S}^c(P^i)=\{A_i\}$$
and, consequently,  
$${\mathcal S}(g_1(Q))\cap {\mathcal S}(g_1(P^i))=\{f(A_i)\}$$
which implies that
\begin{equation}\label{eq1}
f(A_i)=g_1(Q)+g_1(P^i),\;\;\;i\in\{1,\dots,n\}.
\end{equation}
Since $S_{1,\dots,n-1}=V$ is spanned by all $f(A_i)$, $i\in \{1,\dots,n-1\}$, \eqref{eq1} shows that
\begin{equation}\label{eq2}
g_1(Q)+g_1(P^1)+\dots+g_1(P^{n-1})=V.
\end{equation}
In particular, $g_1(P^1),\dots,g_1(P^{n-1})$ are mutually distinct. 
Then
$${\mathcal S}^c(P^i)\cap {\mathcal S}^c(P^j)=\{P^i+P^j\},$$
$${\mathcal S}^c(g_1(P^i))\cap {\mathcal S}^c(g_1(P^j))=\{f(P^i+P^j)\}$$
which means that 
\begin{equation}\label{eq3}
f(P^i+P^j)=g_1(P^i)+g_1(P^j),\;\;\;i,j\in \{1,\dots,n-1\}, i\ne j.
\end{equation}
The equalities $Q+P^1+\dots+P^{n-1}=V$ and \eqref{eq2} imply the existence of a linear automorphism of $V$ which sends $g_1(Q)$ to $Q$ and $g_1(P^i)$ to $P^i$
for every $i\in \{1,\dots,n-1\}$. Let $g$  be the associated automorphism of $\Gamma_2(V)$.
We replace $f$ by $gf$ and obtain that  $g_1$ leaves fixed $Q$ and all $P^i$ with $i\in \{1,\dots,n-1\}$.
By \eqref{eq1} and \eqref{eq3}, 
$$f(A_i)=A_i$$
for every $i\in \{1,\dots,n-1\}$ and 
$$f(P^i+P^j)=P^i+P^j$$
for any distinct $i,j\in \{1,\dots,n-1\}$. 
Also,  
\begin{equation}\label{eq4}
S_I=Q+\sum_{i\in I}P^i=Q+\sum_{i\in I}P_i
\end{equation}
for every non-empty $I\subset \{1,\dots,n-1\}$.

\begin{lemma}\label{lemma2}
The restriction of $f$ to ${\mathcal A}$ is identity.
\end{lemma}

\begin{proof}
We have $f({\mathcal A})={\mathcal A}$ (since $g_1(Q)=Q$).
Show that $f(A_I)=A_I$ for every non-empty $I\subset \{1,\dots,n-1\}$ by induction.
The statement holds if $|I|=1$. 
Suppose that $m$ is an integer such that $1\le m<n-1$ and $f(A_I)=A_I$ for every $I\subset \{1,\dots,n-1\}$ containing at most $m$ elements.
Consider an $(m+1)$-element subset  $I\subset \{1,\dots,n-1\}$.
The equality \eqref{eq4} shows that $A_J$, $J\subset \{1,\dots,n-1\}$, is contained in $S_I$ if and only if $J\subset I$.
By Lemma \ref{lemma1}, $f(A_I)\subset S_I$. 
Since $f(A_J)=A_J$ for every proper $J\subset I$, we obtain that $f(A_I)=A_I$.
\end{proof}

\begin{lemma}\label{lemma3}
If $P_I\in{\mathcal G}'_1(V)$, then $g_1(P_I)=P_I$ or $g_1(P_I)=P^I$.
\end{lemma}

\begin{proof}
We have $A_I\in {\mathcal S}^c(P_I)$ and $$A_I=f(A_I)\in {\mathcal S}(g_1(P_I))$$
which implies that $g_1(P_I)\subset A_I$.
The $1$-dimensional subspaces of $A_I$ are $Q,P_I,P^I$.
It was noted above that $g_1(P_I)\ne Q$ and, consequently, $g_1(P_I)$ is $P_I$ or $P^I$. 
\end{proof}

\begin{lemma}\label{lemma-n4}
The statement of Theorem \ref{theorem-main} holds for $n=4$.
\end{lemma}

\begin{proof}If $n=4$, then
$${\mathcal B}=\{P^1+P^2,\;P^1+P^3,\;P^2+P^3\}$$
and 
$${\mathcal C}={\mathcal S}^c(P^4)=\{P^1+P^4,\; P^2+P^4,\; P^3+P^4\},$$
see Example \ref{exmp1}.
The restriction of $f$ to ${\mathcal A}\cup {\mathcal B}$ is identity.
By Lemma \ref{lemma3}, 
$$g_1(P^4)=P^4\;\mbox{ or }\;g_1(P^4)=P_4.$$
If $i\in \{1,2,3\}$, then $g_1(P^i)=P^i$ and
$$f(P^i+P^4)=g_1(P^i)+g_1(P^4)= P^i+g_1(P^4).$$
If $g_1(P^4)=P^4$, then $f$ is identity. 
Example \ref{exmp2} shows that $f=h$ if $g_1(P^4)=P_4$.
\end{proof}

Let $I$ be a non-empty subset of $\{1,\dots,n-1\}$ and 
let ${\mathcal C}_I$ be the set formed by all elements of ${\mathcal C}(n,2)_2$ contained in $S_I$.
Denote by $\Gamma_I$ the subgraph of $\Gamma(n,2)_2$ induced by ${\mathcal C}_I$.
If $|I|=1$ or $|I|=2$, then $\Gamma_I$ is a single vertex or $K_4$, respectively.
In the case when $|I|\ge 3$, we consider the projection of $V$ on ${\mathbb F}^{|I|+1}$ which 
removes each $i$-coordinate with $i\in \{1,\dots,n-1\}\setminus I$; 
it induces a natural isomorphism between $\Gamma_I$ and $\Gamma(|I|+1,2)_2$.
The restriction of $f$ to ${\mathcal C}_I$ is denoted by $f_I$.
If $I=\{i_1,\dots,i_m\}$, then we write ${\mathcal C}_{i_1,\dots,i_m}$ and $f_{i_1,\dots,i_m}$ instead of ${\mathcal C}_I$ and $f_I$, respectively.

\begin{lemma}\label{lemma4}
The map $f_I$ is an automorphism of  the graph $\Gamma_I$.
\end{lemma}

\begin{proof}
Show that $f(X)\subset S_I$ for every $X\in {\mathcal C}_I$.
This is clear if $X=A_J$, $J\subset I$. Suppose that $X$ is distinct from any such $A_J$.
There are distinct $A_J,A_{T}$, $J,T\subset I$ adjacent to $X$. 
Then $f(X)$ is adjacent to $f(A_J)=A_J$ and $f(A_T)=A_T$. Since $f(X)\not\in {\mathcal A}$,
it  does not contain $Q$ and, consequently,
 intersects $A_J,A_T$ in distinct $1$-dimensional subspaces.
 Therefore,
$$f(X)\subset A_J+A_T\subset S_I$$
which gives the claim.
\end{proof}

\begin{lemma}\label{lemma5}
If $f_I$ is identity for a certain $3$-element $I\subset \{1,\dots,n-1\}$, then $f$ is identity.
\end{lemma}

\begin{proof}
We prove the statement by induction. The case $n=4$ is trivial. Suppose that $n\ge 5$ and $f_I$ is identity for a $3$-element $I\subset \{1,\dots,n-1\}$.
 
Let $J$ and $J'$ be $(n-2)$-element subsets of $\{1,\dots,n-1\}$ such that $I\subset J$ and $I\not\subset J'$.
If $n=5$, then $J=I$. If $n\ge 6$, then $f_J$ is identity by inductive hypothesis (since $\Gamma_J$ can be naturally identified with $\Gamma(n-1,2)_2$).
Show that $f_{J'}$ is identity.
Without loss of generality, we can consider the case when
$$J=\{1,\dots,n-2\}\;\mbox{ and }\; J'=\{2,\dots,n-1\}.$$
If $n\ge 6$, then $J\cap J'$ contains the $3$-element subset $\{2,3,4\}$. 
The map $f_{2,3,4}$ is identity (since $f_J$ is identity) which implies that $f_{J'}$ is identity by inductive hypothesis.
If $n=5$, then 
$$J\cap J'=\{2,3\}.$$
The intersection of 
$$X=P^1+P_{1,2,3}\in {\mathcal C}_{1,2,3}\;\mbox{ and }\;Y=P^{4}+P_{2,3,4}\in {\mathcal C}_{2,3,4}$$
is $P_{1,4}$, i.e. $X,Y$ are adjacent.
Consequently, $f(X)=X$ and $f(Y)$ are adjacent. 
Since
$$X\subset S_{1,2,3},\;\;f(Y)\subset S_{2,3,4},\;\; S_{1,2,3}\cap S_{2,3,4}=S_{2,3}$$
and $X$ intersects $S_{2,3}$ precisely in $P_{1,4}$, the intersection of $X$ and $f(Y)$ is $P_{1,4}$.
Also, $f(Y)$ contains $P^4$
(since $Y\in {\mathcal S}^c(P^4)$ and $g_1(P^4)=P^4$).
Then
$$f(Y)=P^4+P_{1,4}=Y.$$
By Lemma \ref{lemma-n4}, $f_{2,3,4}$ is identity or it is the restriction of $h$ to ${\mathcal C}_{2,3,4}$. 
In the second case, $f(Y)=Y^c$ (since $Y$ belongs to ${\mathcal C}$).
Therefore, $f_{2,3,4}$ is identity.

By the arguments given above, $f_J$ is identity for every $(n-2)$-element subset $J\subset \{1,\dots,n-1\}$.
Observe that ${\mathcal C}(n,2)_2$ is the union of ${\mathcal S}^c(P^n)$ and all ${\mathcal C}_J$ such that $J$ is an $(n-2)$-element subset of $\{1,\dots,n-1\}$.
Thus, we need to show that $f(X)=X$ for every $X\in {\mathcal S}^c(P^n)$ distinct from $A_n$ (we have $f(A_n)=A_n$ by Lemma \ref{lemma2}). 
Since $Q\not\subset X$ and the intersection of all ${\mathcal C}_J$ such that $J$ is an $(n-2)$-element subset of $\{1,\dots,n-1\}$ is $Q$,
there are $(n-2)$-element $J,J'\subset \{1,\dots,n-1\}$ such that $S_J$ and $S_{J'}$ intersect $X$ in distinct $1$-dimensional subspaces $P$ and $P'$, respectively.  
There are the following possibilities:
\begin{enumerate}
\item[(a)] $P,P'\in {\mathcal G}'_1(V)$;
\item[(b)] one of $P,P'$ does not belong to ${\mathcal G}'_1(V)$.
\end{enumerate}
Consider the case (a). 
We have $P=P_{T}$ and $|T|\ge 3$.
The intersection of $J$ and $T$ is non-empty.
If $i\in T\cap J$, then $P+P^i$ is an element of  ${\mathcal C}_J$ containing $P$.
Other element of ${\mathcal C}_J$ containing $P$ is $A_T$.
Then
$$g_1(P)=f(P+P^i)\cap f(A_T).$$
The map $f_J$ is identity which implies that 
$$g_1(P)=(P+P^i)\cap A_T=P.$$
Similarly, we establish that $g_1(P')=P'$.
Hence
$$f(X)=g_1(P)+g_1(P')=P+P'=X.$$
Since $P^n$ is the intersection of $X$ and $A_n$, 
$$g_1(P^n)=f(X)\cap f(A_n)=X\cap A_n=P^n.$$
In the case (b), we obtain that $X=P^i+P^n$ for a certain $i\in \{1,\dots,n-1\}$ and 
$$f(X)=g_1(P^i)+g_1(P^n)=P^i+P^n=X.$$
So, $f$ is identity.
\end{proof}

We assert that $f=h$ if $f$ is non-identity. 
This statement will be proved by induction. It holds for $n=4$ (Lemma \ref{lemma-n4}).
Suppose that $n\ge 5$. Lemma \ref{lemma5} shows that for every $3$-element $I\subset \{1,\dots,n-1\}$ the map $f_I$ is non-identity.
Then  for every $(n-2)$-element $I\subset \{1,\dots,n-1\}$ the map $f_I$ is non-identity and, by inductive hypothesis, $f_I$  is the restriction of $h$ to ${\mathcal C}_I$.

As in the proof of Lemma \ref{lemma5}, we take $X\in {\mathcal S}^c(P^n)$ distinct from $A_n$. 
We will need the following observations:
\begin{enumerate}
\item[$\bullet$] $P^n\subset H$ if $n$ is even;
\item[$\bullet$] $P^n\not\subset H$ if $n$ is odd. 
\end{enumerate}
Consider $(n-2)$-element subsets $J,J'\subset \{1,\dots,n-1\}$ such that $S_J$ and $S_{J'}$ intersect $X$ in distinct $1$-dimensional subspaces $P$ and $P'$, respectively.
One of the following possibilities is realized:
\begin{enumerate}
\item[(a)] $P,P'\in {\mathcal G}'_1(V)$;
\item[(b)] one of $P,P'$ does not belong to ${\mathcal G}'_1(V)$.
\end{enumerate}
Consider the case (a). 
Suppose that $P=P_T$ and $P'=P_{T'}$. Then $P_{T'}\ne P^T$ (otherwise, $X$ contains $Q$ which contradicts our assumption).
There are $Y\in {\mathcal C}_J$ and $Y'\in {\mathcal C}_{J'}$ such that
$$P=Y\cap A_T,\;\;P'=Y'\cap A_{T'}$$
(see the proof of Lemma  \ref{lemma5}). Then
$$g_1(P)=f(Y)\cap A_T,\;\;g_1(P')=f(Y')\cap A_{T'}.$$
Since $f_J$ and $f_{J'}$ are the restriction of $h$ to ${\mathcal C}_J$ and ${\mathcal C}_{J'}$ (respectively), we have
$$g_1(P)=P\;\mbox{ if }\;P\subset H\;\mbox{ and }\; g_1(P)=P^T\;\mbox{ if }\;P\not\subset H,$$
$$g_1(P')=P'\;\mbox{ if }\;P'\subset H\;\mbox{ and }\; g_1(P')=P^{T'}\;\mbox{ if }\;P'\not\subset H.$$
If $X\in {\mathcal B}$, then $P,P',P^n$ are contained in $H$ which means that $n$ is even,
$$f(X)=g_1(P)+g_1(P')=P+P'=X$$
and 
$$g_1(P^n)=f(X)\cap f(A_n)=X\cap A_n=P^n.$$
Suppose that $X\in {\mathcal C}$. Then at least one of $P,P'$ is not contained in $H$.
If $P,P'$ both are not contained in $H$, then $n$ is even (since $P^n\subset H$) and
$$f(X)=g_1(P)+g_1(P')=P^T+P^{T'}=X^c;$$
since $P^n$ is contained in $H$, we have $P^n\subset X^c$ which implies that  
$$g_1(P^n)=f(X)\cap f(A_n)= X^c\cap A_n=P^n.$$
If $P\not\subset H$ and $P'\subset H$, then $n$ is odd (since $P^n\not\subset H$) and 
$$f(X)=g_1(P)+g_1(P')=P^T+P'=X^c;$$
in this case, $X^c$ contains $P_n$  and
$$g_1(P^n)=f(X)\cap f(A_n)= X^c\cap A_n=P_n.$$
The case when $P\subset H$ and $P'\not\subset H$ is similar.
So,
$$f(X)=X\;\mbox{ if }\;X\in {\mathcal B}\;\mbox{ and }\;f(X)=X^c\;\mbox{ if }\;X\in {\mathcal C}.$$
Furthermore, we established the following:
\begin{enumerate}
\item[$\bullet$] $g_1(P^n)=P^n$ if $n$ is even,
\item[$\bullet$] $g_1(P^n)=P_n$ if $n$ is odd.
\end{enumerate}
In the case (b), 
we have $X=P^i+P^n$ for a certain $i\in \{1,\dots,n-1\}$.
If $n$ is even, then $X\in {\mathcal B}$ and
$$f(X)=g_1(P^i)+g_1(P^n)=P^i+P^n=X.$$
If $n$ is odd, then $X\in {\mathcal C}$ and
$$f(X)=g_1(P^i)+g_1(P^n)=P^i+P_n=X^c.$$
So, $f=h$.

\end{document}